\documentclass[a4paper]{amsart}
\usepackage{amssymb}
\usepackage{verbatim}
\usepackage[dvipdfmx]{graphicx}
\usepackage{tikz-cd}
\usepackage{color}
\usepackage{here}
\theoremstyle{plain}
  \newtheorem{thm}{Theorem}[section]
  \newtheorem{lem}[thm]{Lemma}
  \newtheorem{cor}[thm]{Corollary}
  \newtheorem{prop}[thm]{Proposition}
  \newtheorem{conj}[thm]{Conjecture}

  \newtheorem*{thm*}{Theorem}
  \newtheorem*{obs*}{Observation}
\theoremstyle{definition}
  \newtheorem{defn}[thm]{Definition}

\theoremstyle{remark}
  \newtheorem{rem}[thm]{Remark}

\newcommand{\Z}{\mathbb{Z}}
\newcommand{\C}{\mathbb{C}}
\newcommand{\R}{\mathbb{R}}

\newcommand{\Vol}{\operatorname{Vol}}

\newcommand{\CS}{\operatorname{CS}}
\newcommand{\Li}{\operatorname{Li}}
\newcommand{\Hom}{\operatorname{Hom}}
\newcommand{\cs}{\operatorname{cs}}

\newcommand{\SL}{\mathrm{SL}}
\newcommand{\PSL}{\mathrm{PSL}}

\newcommand{\arccosh}{\operatorname{arccosh}}

\newcommand{\Int}{\operatorname{Int}}
\newcommand{\floor}[1]{\lfloor#1\rfloor}

\newcommand{\vE}{v_{E}}

\newcommand{\FigEight}{\mathcal{E}}
\newcommand{\Cable}{\mathcal{C}}
\newcommand{\FigEightCable}{\mathcal{F}}
\renewcommand{\i}{\sqrt{-1}}
\newcommand{\pic}[2]{\raisebox{-0.5\height}{\includegraphics[scale=#1]{#2.eps}}}
\numberwithin{equation}{section}
\allowdisplaybreaks
\begin{document}
\title[The colored Jones polynomial of a cable of the figure-eight knot]
{The colored Jones polynomial of a cable of the figure-eight knot}
\author{Hitoshi Murakami}
\address{
Graduate School of Information Sciences,
Tohoku University,
Aramaki-aza-Aoba 6-3-09, Aoba-ku,
Sendai 980-8579, Japan}
\email{starshea@tky3.3web.ne.jp}
\author{Anh T.~Tran}
\address{
Department of Mathematical Sciences, The University of Texas at Dallas, Richardson,
TX 75080, USA}
\email{att140830@utdallas.edu}
\date{\today}
\begin{abstract}
We study the asymptotic behavior of the $N$-dimensional colored Jones polynomial of a cable of the figure-eight knot, evaluated at $\exp(\xi/N)$ for a real number $\xi$.
We show that if $\xi$ is sufficiently large, the colored Jones polynomial grows exponentially when $N$ goes to the infinity.
Moreover the growth rate is related to the Chern--Simons invariant of the knot exterior associated with an $\SL(2;\R)$ representation.
\end{abstract}
\dedicatory{Dedicated to the memory of Vaughan Jones (1952--2020)}
\thanks{The authors are supported by Grant-in-Aid for Challenging Exploratory Research (21654053).
The first author is supported by JSPS KAKENHI Grant Number 20K03601.
The second author is supported by a grant from the Simons Foundation (\#708778).}
\maketitle
\section{Introduction}
For a knot $K$ in the three-sphere $S^3$, let $J_N(K;q)$ be the $N$-dimensional colored Jones polynomial \cite{Jones:BULAM31985,Kirillov/Reshetikhin:1989}.
Here we normalize $J_N(K;q)$ so that $J_N(U;q)=1$ for the unknot $U$, and that when $N=2$, it satisfies the following skein relation:
\begin{equation*}
  qJ_2\left(\raisebox{1mm}{\pic{0.1}{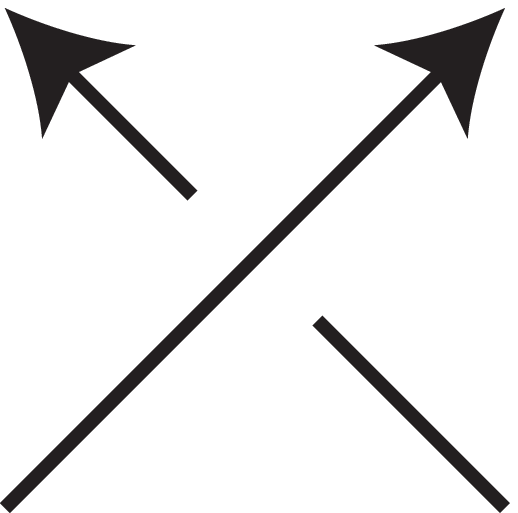}};q\right)
  -
  q^{-1}J_2\left(\raisebox{1mm}{\pic{0.1}{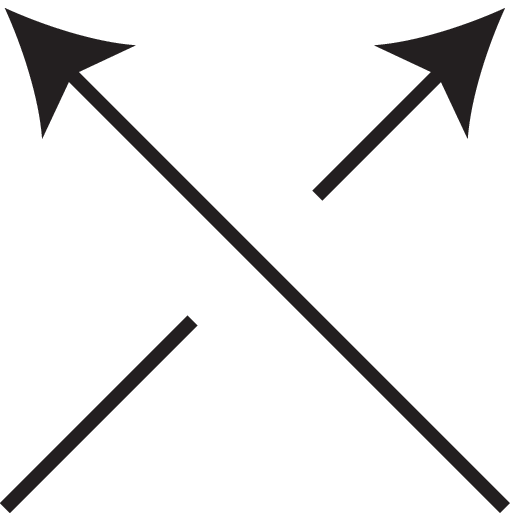}};q\right)
  =
  \left(q^{1/2}-q^{-1/2}\right)q_2\left(\raisebox{1mm}{\pic{0.1}{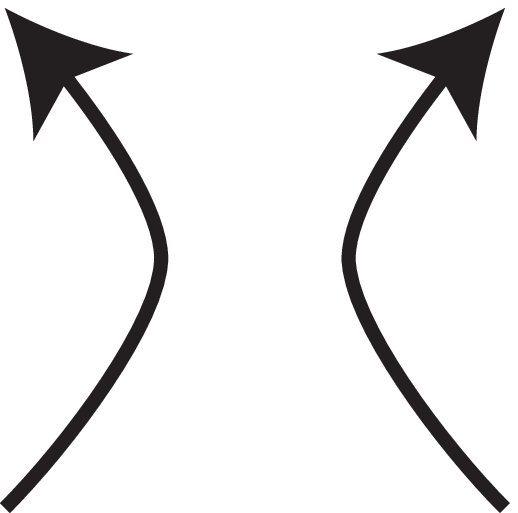}};q\right).
\end{equation*}
Note that it is different from that of Jones' original paper \cite{Jones:BULAM31985}.
\par
R.~Kashaev \cite{Kashaev:LETMP97} proposed a conjecture stating that his invariant parametrized by an integer $N\ge2$ introduced in \cite{Kashaev:MODPLA95} grows exponentially with growth rate proportional to the hyperbolic volume of the knot complement for any hyperbolic knot.
J.~Murakami and the first author \cite{Murakami/Murakami:ACTAM12001} proved that Kashaev's invariant coincides with $J_N(K;\exp(2\pi\i/N))$ and generalized Kashaev's conjecture for general knots.
\begin{conj}[Volume Conjecture]
Let $K$ be a knot.
Then we have
\begin{equation}\label{eq:VC}
  \lim_{N\to\infty}\frac{\log|J_N(K;\exp(2\pi\i/N))|}{N}
  =
  \frac{\Vol(S^3\setminus{K})}{2\pi},
\end{equation}
where $\Vol$ is $v_3$ times the simplicial volume.
Here $v_3$ is the volume of the ideal, hyperbolic, regular tetrahedron.
\end{conj}
Note that when $K$ is hyperbolic, that is, $S^3\setminus{K}$ possesses a unique complete hyperbolic structure with finite volume, then $\Vol$ is just the hyperbolic volume.
So the volume conjecture is a generalization of Kashaev's conjecture.
\par
It is well known that the knot complement $S^3\setminus{K}$ can be decomposed into hyperbolic pieces and Seifert fibered pieces by a system of tori (Jaco--Shalen--Johannson decomposition \cite{Jaco/Shalen:MEMAM1979,Johannson:1979}) and that $\Vol(S^3\setminus{K})$ is the sum of the hyperbolic volumes of the hyperbolic pieces.
Hence if there are no hyperbolic pieces, then $\Vol(S^3\setminus{K})=0$.
In fact, Kashaev and O.~Tirkkonen \cite{Kashaev/Tirkkonen:ZAPNS2000} proved that when $K$ is a torus knot, the limit in \eqref{eq:VC} is zero, proving the volume conjecture.
\par
For hyperbolic knots, T.~Ekholm showed that the volume conjecture is true for the figure-eight knot $4_1$.
The conjecture is also proved for $5_2$ by T.~Ohtsuki \cite{Ohtsuki:QT2016}, and for $6_1,6_2,6_3$ by Ohtsuki and Y.~Yokota \cite{Ohtsuki/Yokota:MATPC2018}.
\par
For non-hyperbolic knots with non-zero volumes, H.~Zheng \cite{Zheng:CHIAM22007} proved the volume conjecture in the case of Whitehead doubles of the torus knot of type $(2,2b+1)$, and L{\^e} and the second author proved it for cables of the figure-eight knot \cite{Le/Tran:JKNOT2010}.
\par
What happens when we replace $2\pi\i$ in $J_N\left(K;e^{2\pi\i/N}\right)$ with another complex number $\eta$?
\par
In the case of the figure-eight knot $E$, the following results about the asymptotic behavior of $J_N\left(E;e^{\eta/N}\right)$ ($N\to\infty$) are known so far.
\begin{enumerate}
\item
$|2\pi\i-\eta|$ is small enough.
If $\eta$ is not purely imaginary, then the limit
\begin{equation*}
  \lim_{N\to\infty}
  \frac{\log J_N(E;e^{\eta/N})}{N}
\end{equation*}
exists and it determines the volume and the Chern--Simons invariant of the three-manifold obtained by generalized Dehn surgery corresponding to $\eta-2\pi\i$ \cite{Murakami/Yokota:JREIA2007} (see also \cite{Murakami:JTOP2013}).
\item
$\eta$ is purely imaginary.
If $2\pi/|\eta|$ is irrational, its irrationality measure is finite, and $5\pi/3<|\eta|<7\pi/3$, then
\begin{equation*}
  \lim_{N\to\infty}\frac{\log\left|J_N\left(E;e^{\eta/N}\right)\right|}{N}
  =
  \frac{\Vol\left(E_{|\eta-2\pi\i|}\right)}{\eta},
\end{equation*}
where $E_{\theta}$ is the cone-manifold along $E$ with cone-angle $\theta$ \cite[Theorem~1.2]{Murakami:KYUMJ2004} (see \cite{Murakami:KYUMJ2016} for a correction).
\item
$\eta$ is real.
If $\eta>2\kappa$ with $\kappa:=\arccosh(3/2)/2=\log\bigl((3+\sqrt{5})/2\bigr)/2$, then
\begin{equation}\label{eq:fig8_kappa}
  \lim_{N\to\infty}\frac{\log\left|J_N\left(E;e^{\eta/N}\right)\right|}{N}
  =
  \frac{1}{\eta}S(\eta/2),
\end{equation}
where
\begin{equation*}
  S(\xi)
  :=
  \Li_2\left(e^{-\varphi(\xi)-2\xi}\right)
  -
  \Li_2\left(e^{\varphi(\xi)-2\xi}\right)
  +
  2\xi\varphi(\xi)
\end{equation*}
with $\varphi(\xi):=\arccosh(\cosh(2\xi)-1/2)$ and $\Li_2(z):=-\int_{0}^{z}\frac{\log(1-x)}{x}\,dx$.
See \cite[Theorem~8.1]{Murakami:KYUMJ2004} and \cite[Theorem~6.9]{Murakami:Novosibirsk}.
\item
$\eta=2\kappa$.
In this case $J_N\left(E;e^{\eta/N}\right)$ grows polynomially with respect to $N$.
Moreover we have
\begin{equation*}
  J_N\left(E;e^{2\kappa/N}\right)
  \underset{N\to\infty}{\sim}
  \frac{\Gamma(1/3)}{(6\kappa)^{2/3}}N^{2/3},
\end{equation*}
where $\Gamma(x)$ is the gamma function, and $f(N)\underset{N\to\infty}{\sim}g(N)$ means $\lim_{N\to\infty}\frac{f(N)}{g(N)}=1$ (\cite[Theorem~1.1]{Hikami/Murakami:COMCM2008}).
\item
$|\eta|$ is small enough.
Then $\lim_{N\to\infty}\frac{\log\left|J_N\left(E;e^{\eta/N}\right)\right|}{N}=0$.
Moreover we have
\begin{equation*}
  \lim_{N\to\infty}J_N\left(E;e^{\eta/N}\right)
  =
  \frac{1}{\Delta(E;e^{\eta})},
\end{equation*}
where $\Delta(K;t)$ is the Alexander polynomial of a knot $K$ normalized so that $\Delta(K;1)=1$ and $\Delta(K;t^{-1})=\Delta(K;t)$ (\cite{Murakami:JPJGT2007}).
This is also true for any knot \cite[Theorem~1.3]{Garoufalidis/Le:GEOTO2011}.
\end{enumerate}
See \cite{Gukov/Murakami:FIC2008,Dimofte/Gukov:CS,Dimofte/Gukov/Lenells/Zagier:CNTP2010} for more results.
\par
In this paper, we study the colored Jones polynomial of the $(2,2b+1)$-cable of the figure-eight knot (Figure~\ref{fig:fig8_cable}), which we denote by $E^{(2,2b+1)}$.
\begin{figure}[H]
\pic{0.4}{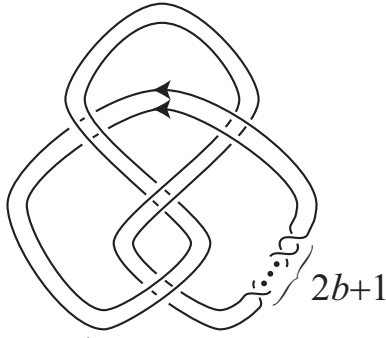}
\caption{$(2,2b+1)$-cable of the figure-eight knot $E^{(2,2b+1)}$}
\label{fig:fig8_cable}
\end{figure}
We put $\kappa:=\arccosh(3/2)/2$ as before.
We also put $\FigEightCable:=S^3\setminus\Int{N(E^{(2,2b+1)})}$, where $N(E^{(2,2b+1)})$ is a regular neighborhood of $E^{(2,2b+1)}$ in $S^3$ and $\Int$ denotes the interior.
We prove
\begin{thm*}[Theorem~\ref{thm:main}]
Let $\xi$ be a real number with $\xi>\kappa$, then we have
\begin{equation*}
  \xi\lim_{N\to\infty}
  \frac{\log J_{N}\left(E^{(2,2b+1)};e^{\xi/N}\right)}{N}
  =
  S(\xi),
\end{equation*}
where we put
\begin{equation*}
  S(\xi)
  :=
  \Li_2\left(e^{-\varphi(\xi)-2\xi}\right)
  -
  \Li_2\left(e^{\varphi(\xi)-2\xi}\right)
  +2\xi\varphi(\xi)
\end{equation*}
with
\begin{equation*}
  \varphi(\xi)
  :=
  \arccosh\left(\cosh(2\xi)-\frac{1}{2}\right).
\end{equation*}
Moreover, $S(\xi)$ defines the Chern--Simons invariant $\CS_{\FigEightCable}([\rho_{\xi}];\xi,v(\xi))$ of the knot exterior $\FigEightCable$ associated with a representation $\rho_{\xi}\colon\pi_1(\FigEightCable)\to\SL(2;\C)$ and $(\xi,v(\xi))$ {\rm(}see Section~\ref{sec:CS} for the definition{\rm)} in the following way:
\begin{equation*}
  \CS_{\FigEightCable}([\rho_{\xi}];\xi,v(\xi))
  =
  S(\xi)-\frac{\xi v(\xi)}{4}.
\end{equation*}
Here $\rho_{\xi}$ sends the meridian of $\partial\FigEightCable$ to $\begin{pmatrix}e^{\xi/2}&0\\0&e^{-\xi/2}\end{pmatrix}$ and the longitude to $\begin{pmatrix}e^{v(\xi)/2}&0\\0&e^{-v(\xi)/2}\end{pmatrix}$ up to conjugation.
\end{thm*}
Compare this with the following, which was proved in \cite[Theorem~8.1]{Murakami:KYUMJ2004} and \cite[Theorem~6.9]{Murakami:Novosibirsk} (see \eqref{eq:fig8_kappa}).
We put $\FigEight=S^3\setminus\Int{N(E)}$.
\begin{thm*}[Theorem~\ref{thm:fig8}]
Let $\eta$ be a real number with $\eta>2\kappa$, then we have
\begin{equation*}
  \eta\lim_{N\to\infty}\frac{J_N(E;e^{\eta/N})}{N}
  =
  S(\eta/2).
\end{equation*}
Moreover we have
\begin{equation*}
  \CS_{\FigEight}([\sigma_{\eta}];\eta,\vE(\eta))
  =
  S(\eta/2)-\frac{\eta\vE(\eta)}{4},
\end{equation*}
where $\sigma_{\eta}$ sends the meridian of $\partial\FigEight$ to $\begin{pmatrix}e^{\eta/2}&0\\0&e^{-\eta/2}\end{pmatrix}$ and the longitude to $\begin{pmatrix}e^{\vE(\eta)/2}&0\\0&e^{-\vE(\eta)/2}\end{pmatrix}$ up to conjugation.
\end{thm*}
\section{Preliminaries}
In this section, we use linear skein theory (see for example \cite[Chapters~13, 14]{Lickorish:1997}) to calculate the colored Jones polynomial.
For another way to calculate it, see \cite{Le/Tran:JKNOT2010}.
\par
We denote by \pic{0.2}{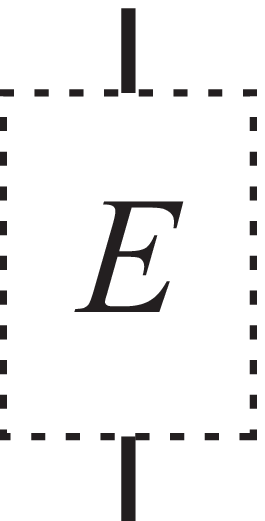} the $(1,1)$-tangle \pic{0.2}{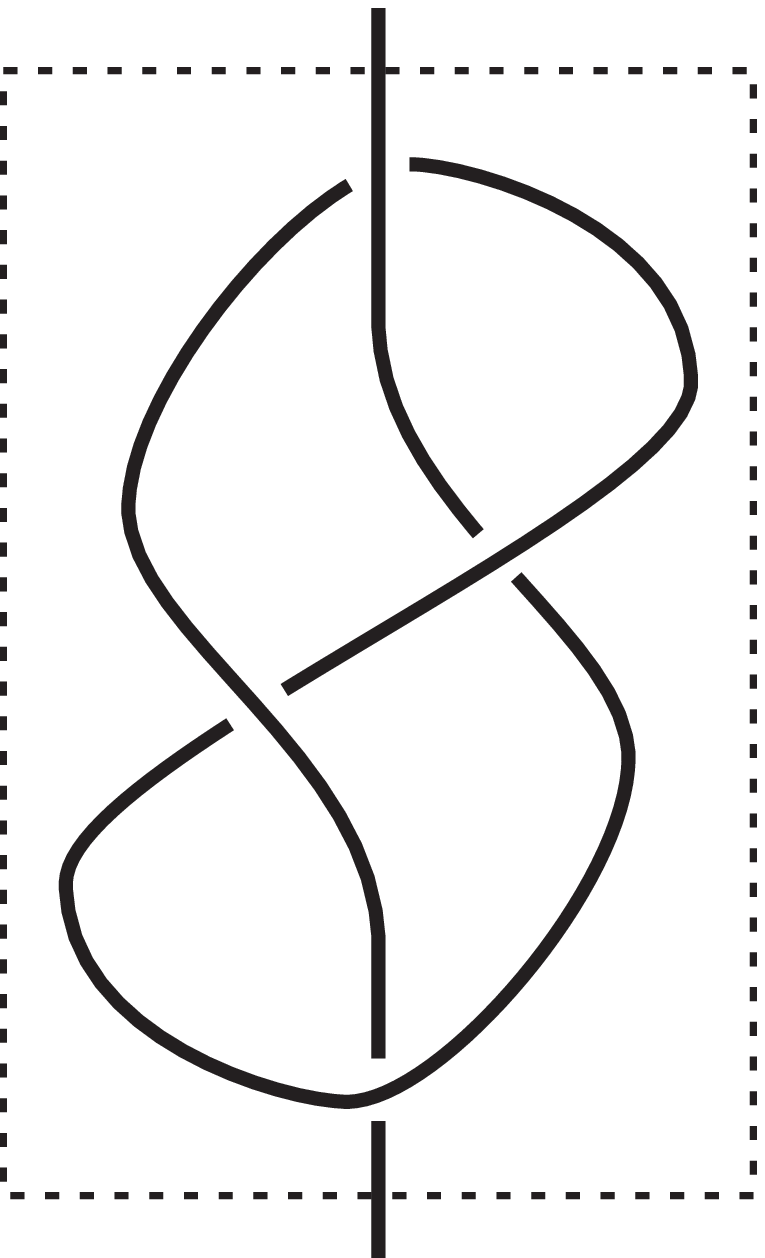}, and by \pic{0.2}{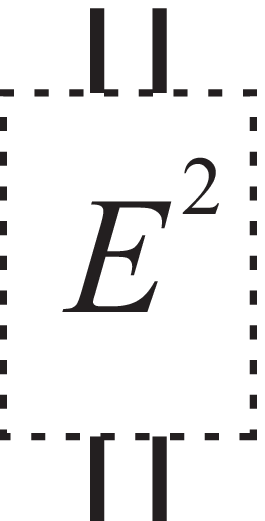} the $(2,2)$-tangle \pic{0.2}{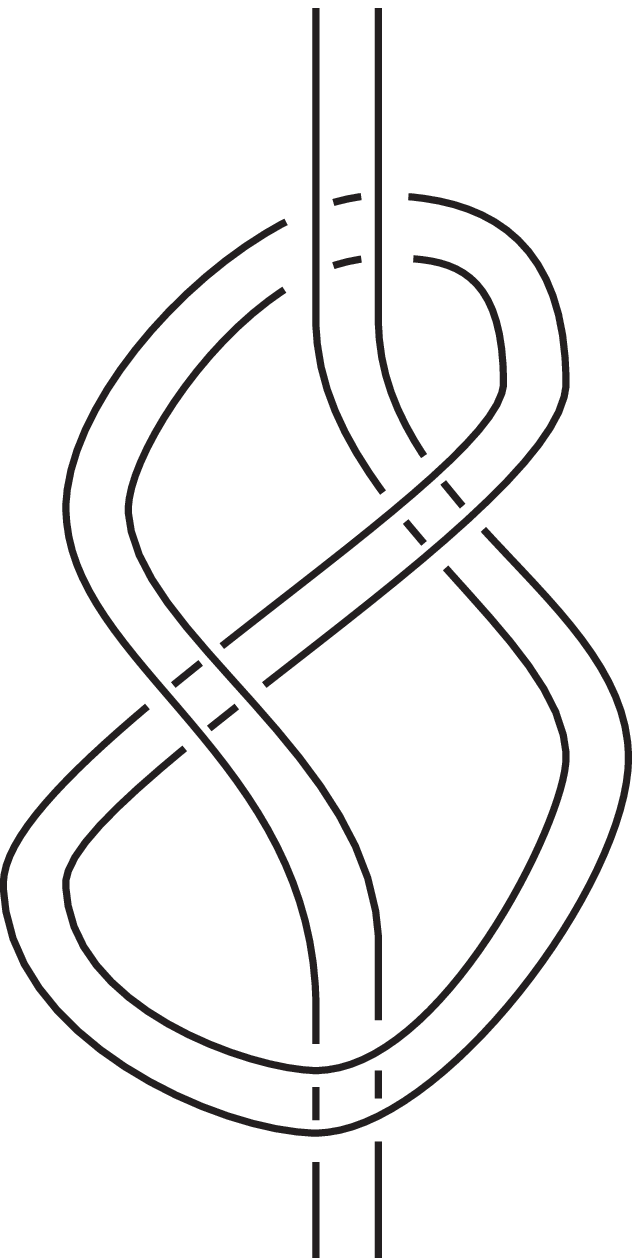} that is the two-parallel of \pic{0.2}{E}.
\par
The $N$-dimensional colored Jones polynomial of $E^{(2,2b+1)}$ is given by
\begin{equation*}
  \left.
    \frac{1}{\Delta_{N-1}}
    \left((-1)^{N-1}A^{(N-1)^2+2(N-1)}\right)^{-(2b+1)}
    \left\langle\pic{0.2}{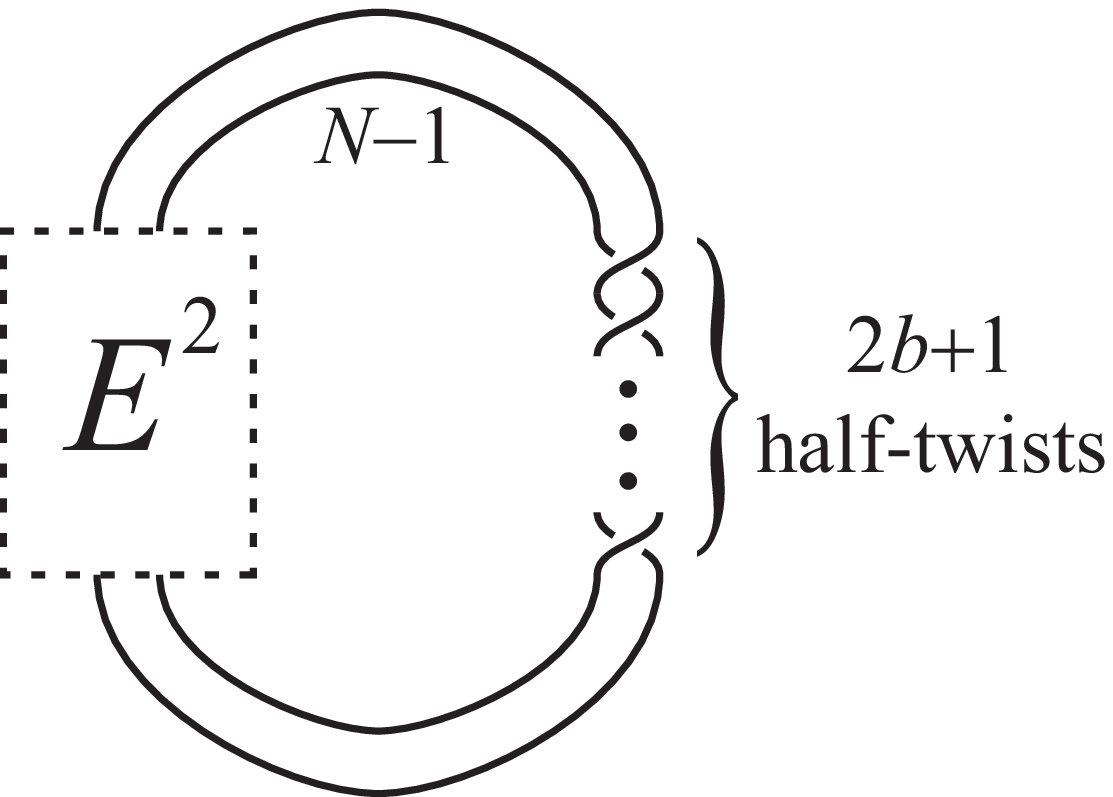}\right\rangle
  \right|_{A:=t^{1/4}},
\end{equation*}
where the number $N-1$ beside a line indicates that we put the $(N-1)$-th Jones--Wenzl indempotent \cite{Jones:INVEM1983,Wenzl:CRMAR1987} along the line, $\langle D\rangle$ is the Kauffman bracket \cite{Kauffman:TOPOL1987} of a link diagram $D$, and $\Delta_{k}:=(-1)^{k}\frac{A^{2(k+1)-A^{-2(k+1)}}}{A^2-A^{-2}}$.
Note that the writhe of \pic{0.2}{E2_2b+1} is $(2b+1)$ and so we need to multiply by $\left((-1)^{N-1}A^{(N-1)^2+2(N-1)}\right)^{-(2b+1)}$ to obtain a unframed knot invariant (see \cite[Lemma~14.1]{Lickorish:1997}).
We use the notation described in \cite[Chapters~13, 14]{Lickorish:1997}.
\par
By using identities in \cite[Chapter~14]{Lickorish:1997}, we have
\begin{equation}\label{eq:linear_skein}
\begin{split}
  &\left\langle
    \pic{0.3}{E2_2b+1}
  \right\rangle
  \\
  =&
  \\
  &\text{(Figure~14.15 in \cite{Lickorish:1997})}
  \\
  &
  \sum_{c\colon\text{even},0\le c\le2(N-1)}
  \frac{\Delta_c}{\theta(N-1,N-1,c)}
  \left\langle
    \pic{0.3}{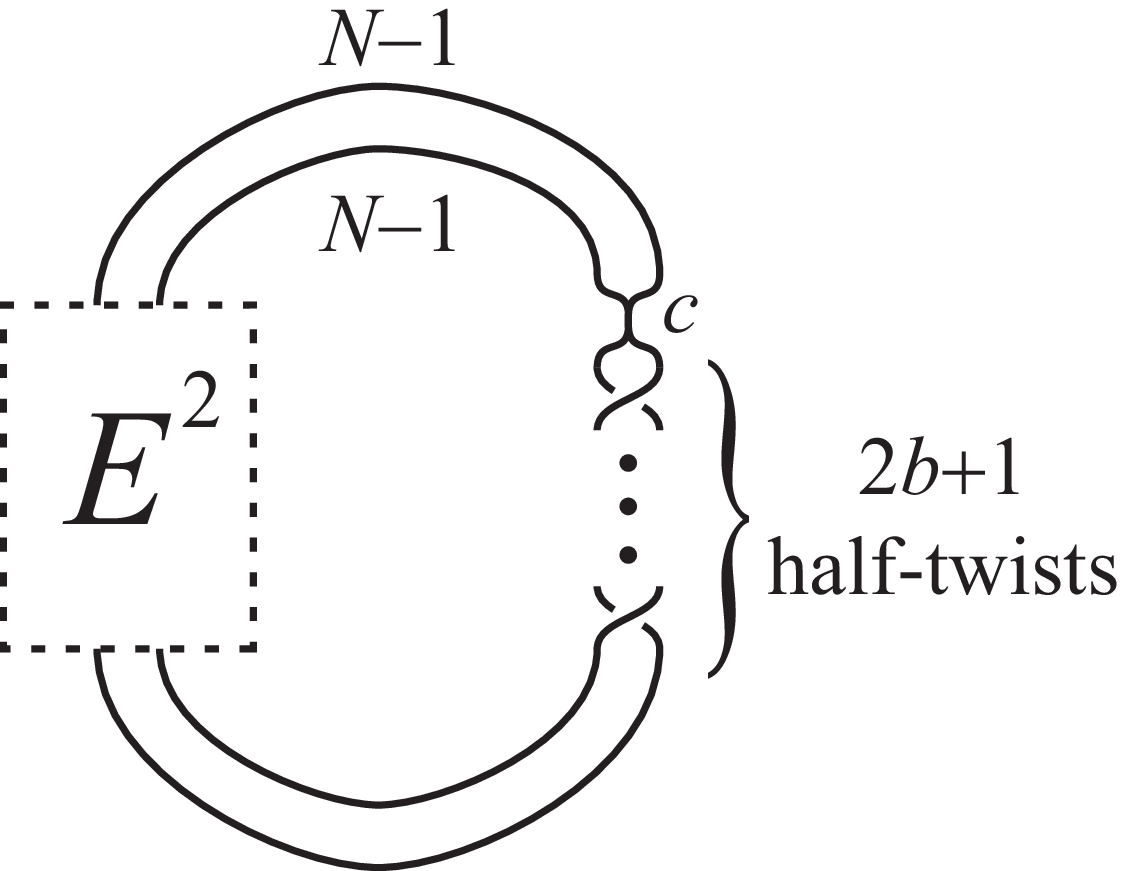}
  \right\rangle
  \\
  =&
  \\
  &\text{(Figure~14.14 in \cite{Lickorish:1997})}
  \\
  &\sum_{c}
  (-1)^{N-1-c/2}A^{(2b+1)(-2N+2+c-(N-1)^2+c^2/2)}
  \frac{\Delta_c}{\theta(N-1,N-1,c)}
  \left\langle
    \pic{0.3}{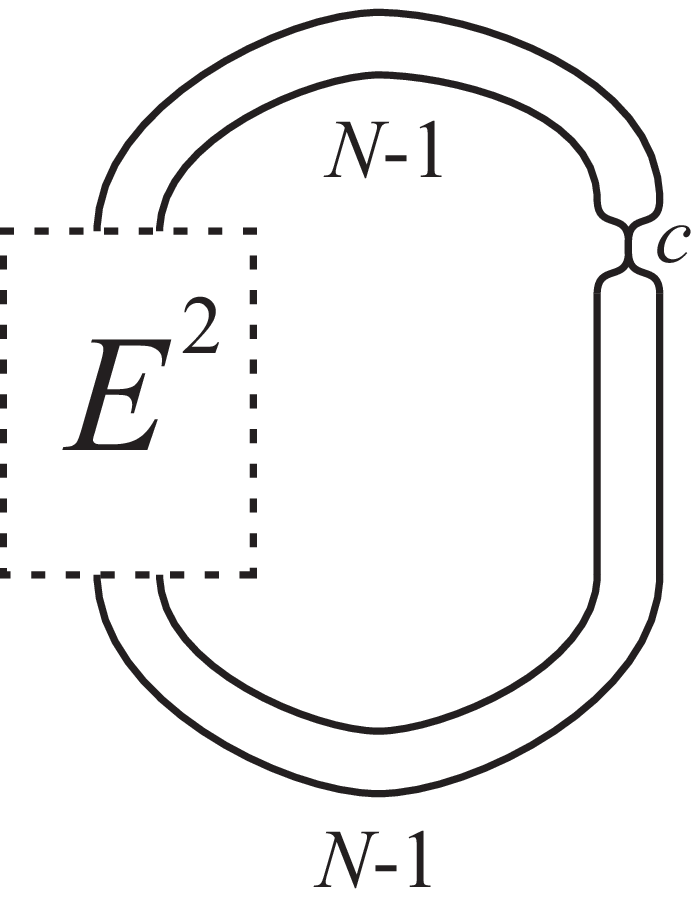}
  \right\rangle
  \\
  =&
  \\
  &\text{(Reidemeister moves II and III)}
  \\
  &
  \sum_{c}
  (-1)^{N-1-c/2}A^{(2b+1)(-N^2+1+c+c^2/2)}
  \frac{\Delta_c}{\theta(N-1,N-1,c)}
  \left\langle
    \pic{0.3}{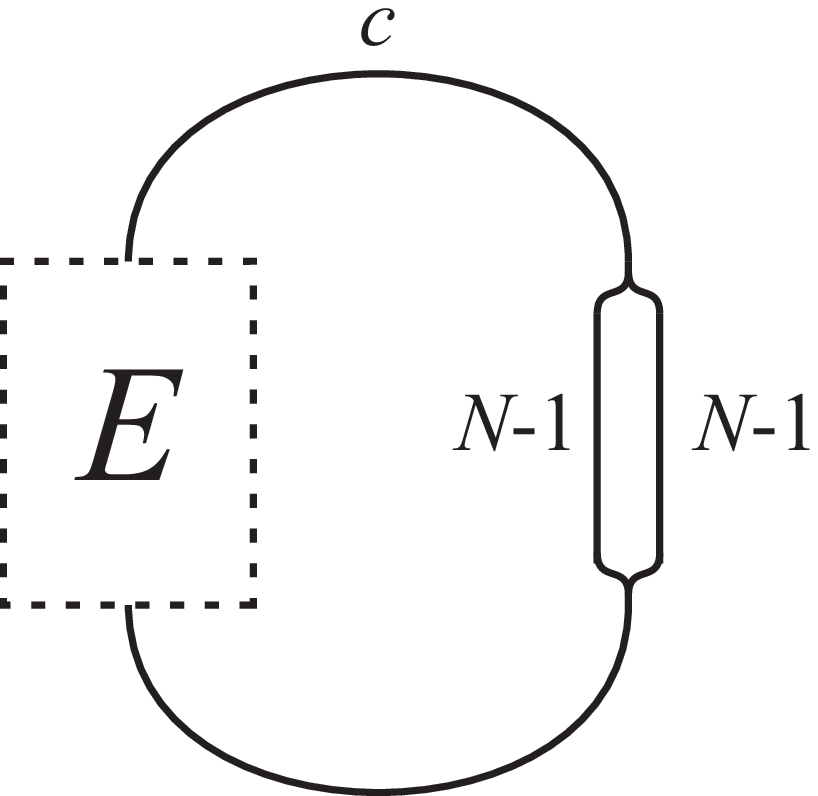}
  \right\rangle
  \\
  =&
  \\
  &\text{(Figure~14.8 in \cite{Lickorish:1997})}
  \\
  &\sum_{c}(-1)^{N-1-c/2}A^{(2b+1)(-N^2+1+c+c^2/2)}
  \left\langle
    \pic{0.3}{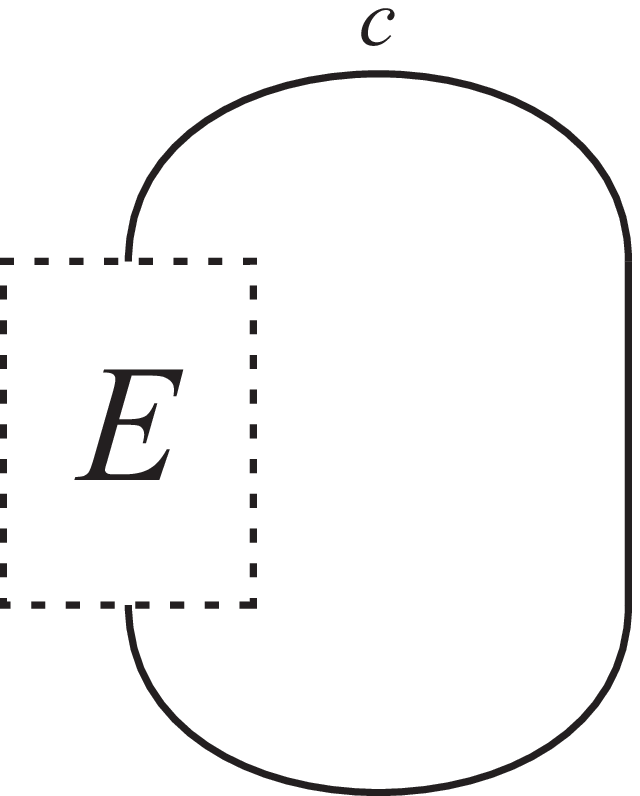}
  \right\rangle.
\end{split}
\end{equation}
\par
Now $\left\langle\pic{0.3}{E_2}\right\rangle$ equals $\left(J_{c+1}(E;t)\Bigm|_{t:=A^{4}}\right)\Delta_{c+1}$.
So we have the following proposition.
\begin{prop}\label{prop:cJ_t}
We have
\begin{equation*}
\begin{split}
  &J_N\left(E^{(2,2b+1)};t\right)
  \\
  =&
  \frac{(-1)^{N-1}t^{-(2b+1)(N^2-1)/2}}{t^{N/2}-t^{-N/2}}
  \sum_{d=0}^{N-1}(-1)^{d}
  t^{(2b+1)(d^2+d)/2}\left(t^{(2d+1)/2}-t^{-(2d+1)/2}\right)
  \\
  &\times
  \sum_{l=0}^{2d}
  \prod_{k=1}^{l}
  \left(t^{(2d+1+k)/2}-t^{-(2d+1+k)/2}\right)
  \left(t^{(2d+1-k)/2}-t^{-(2d+1-k)/2}\right).
\end{split}
\end{equation*}
\end{prop}
\begin{proof}
From \eqref{eq:linear_skein}, we have
\begin{equation*}
\begin{split}
  &(-1)^{N-1}\frac{t^{N/2}-t^{-N/2}}{t^{1/2}-t^{-1/2}}J_N\left(E^{(2,2b+1)};t\right)
  \\
  =&
  ((-1)^{N-1}A^{(N-1)^2+2(N-1)})^{-(2b+1)}
  \\
  &\times
  \left.
    \sum_{\substack{0\le c\le2(N-1)\\ \text{$c$: even}}}
    (-1)^{N-1-c/2}A^{(2b+1)(-N^2+1+c+c^2/2)}
    \left(J_{c+1}(E;t)\Bigm|_{t:=A^{4}}\right)\Delta_{c+1}
  \right|_{A:=t^{1/4}}
    \\
  =&
  \\
  &\text{($c:=2d$)}
  \\
  &
  \sum_{d=0}^{N-1}
  \frac{t^{(2d+1)/2}-t^{-(2d+1)/2}}{t^{1/2}-t^{-1/2}}
  (-1)^{N-1-d}t^{(2b+1)(-N^2+1+2d+2d^2)/4}
  J_{c+1}(E;t)
  \\
  =&
  (-1)^{N-1}\frac{t^{-(2b+1)(N^2-1)/2}}{t^{1/2}-t^{-1/2}}
  \sum_{d=0}^{N-1}
  \left(t^{(2d+1)/2}-t^{-(2d+1)/2}\right)
  (-1)^{d}t^{(2b+1)(d+d^2)/2}
  J_{2d+1}(E;t).
\end{split}
\end{equation*}
Now using the following formula of the $m$-dimensional colored Jones polynomial of the figure-eight knot by K.~Habiro and T.~L{\^e} \cite{Habiro:SURIK2000,Masbaum:ALGGT12003}
\begin{equation*}
  J_{m}(E;t)
  =
  \sum_{l=0}^{m-1}\prod_{k=1}^{l}
  \left(t^{(m+k)/2}-t^{-(m+k)/2}\right)
  \left(t^{(m-k)/2}-t^{-(m-k)/2}\right),
\end{equation*}
we obtain the required formula.
\end{proof}
\section{Limit}
Fix a positive real number $\xi$.
Then from proposition~\ref{prop:cJ_t} we have
\begin{equation*}
  J_N\left(E^{(2,2b+1)};e^{\xi/N}\right)
  =
  \frac{(-1)^{N-1}\exp\left(-\frac{(2b+1)(N^2-1)\xi}{2N}\right)}{2\sinh(\xi/2)}
  \sum_{d=0}^{N-1}\sum_{l=0}^{2d}
  (-1)^{d}f_{d,l}
\end{equation*}
with
\begin{multline*}
  f_{d,l}
  =
  \exp\left(\frac{(2b+1)(d^2+d)\xi}{2N}\right)
  \times
  2\sinh\left(\frac{(2d+1)\xi}{2N}\right)
  \\
  \times
  \prod_{k=1}^{l}4
  \sinh\left(\frac{(2d+1+k)\xi}{2N}\right)
  \sinh\left(\frac{(2d+1-k)\xi}{2N}\right).
\end{multline*}
\begin{lem}
Let $l$ and $d$ be integers with $1\le d\le N-1$ and $0\le l\le2d-2$, we have $f_{d,l}> f_{d-1,l}$.
\end{lem}
\begin{proof}
We first note that $f_{d,l}$ is positive.
We have
\begin{equation}\label{eq:f_d+1_d}
  \frac{f_{d,l}}{f_{d-1,l}}
  =
  \exp\left(\frac{(2b+1)d\xi}{N}\right)
  \frac{\sinh\left(\frac{(2d+1)\xi}{2N}\right)}
       {\sinh\left(\frac{(2d-1)\xi}{2N}\right)}
  \prod_{k=1}^{l}
  \frac{\sinh\left(\frac{(2d+1+k)\xi}{2N}\right)}
       {\sinh\left(\frac{(2d-1+k)\xi}{2N}\right)}
  \frac{\sinh\left(\frac{(2d+1-k)\xi}{2N}\right)}
       {\sinh\left(\frac{(2d-1-k)\xi}{2N}\right)}.
\end{equation}
Since $\xi>0$, we have
\begin{align*}
  \exp\left(\frac{(2b+1)d\xi}{N}\right)
  &>1,
  \\
  \frac{\sinh\left(\frac{(2d+1)\xi}{2N}\right)}
       {\sinh\left(\frac{(2d-1)\xi}{2N}\right)}
  &>1,
  \\
  \frac{\sinh\left(\frac{(2d+1+k)\xi}{2N}\right)}
       {\sinh\left(\frac{(2d-1+k)\xi}{2N}\right)}
  &>1,
  \\
  \frac{\sinh\left(\frac{(2d+1-k)\xi}{2N}\right)}
       {\sinh\left(\frac{(2d-1-k)\xi}{2N}\right)}
  &>1.
\end{align*}
Therefore we have the required inequality.
\end{proof}
\begin{cor}\label{cor:max_f_d}
For any $l$ and $d$ with $0\le d\le N-2$ and $0\le l\le2d$, we have
\begin{equation*}
  f_{d,l}
  <
  f_{N-1,l}.
\end{equation*}
\end{cor}
Define
\begin{equation*}
  S
  :=
  (-1)^{N-1}
  \sum_{d=0}^{N-1}\sum_{l=0}^{2d}
  (-1)^{d}f_{d,l}
\end{equation*}
so that
\begin{equation*}
  J_{N}\left(E^{(2,2b+1)};e^{\xi/N}\right)
  =
  \frac{\exp\left(-\frac{(2b+1)(N^2-1)\xi}{2N}\right)}{2\sinh(\xi/2)}S.
\end{equation*}
Then we have the following lemma.
\begin{lem}\label{lem:S}
The following inequality holds.
\begin{equation*}
  S
  >
  (1-e^{-\xi/2})\sum_{l=0}^{2N-2}f_{N-1,l}.
\end{equation*}
\end{lem}
\begin{proof}
From \eqref{eq:f_d+1_d} and the subsequent inequalities we have
\begin{equation}\label{eq:f_N-1_N-2}
  \frac{f_{N-1,l}}{f_{N-2,l}}
  >
  e^{(2b+1)(N-1)\xi/N}
  \ge
  e^{(N-1)\xi/N}
  \ge
  e^{\xi/2}
\end{equation}
since $b\ge0$ and $N\ge2$.
\par
Suppose that $N$ is even.
We have
\begin{equation*}
\begin{split}
  S
  &=
  \sum_{d=0}^{N-1}\sum_{l=0}^{2d}
  (-1)^{d-1}f_{d,l}
  \\
  &=
  \sum_{k=0}^{(N-2)/2}
  \left(
    -\sum_{l=0}^{4k}f_{2k,l}
    +
    \sum_{l=0}^{4k+2}f_{2k+1,l}
  \right)
  \\
  &=
  \sum_{k=0}^{(N-2)/2}
  \left(
    \sum_{l=0}^{4k}
    \left(f_{2k+1,l}-f_{2k,l}\right)
    +
    f_{2k+1,4k+1}+f_{2k+1,4k+2}
  \right)
\end{split}
\end{equation*}
Since $f_{2k+1,l}-f_{2k,l}>0$, $f_{2k+1,4k+1}>0$, and $f_{2k+1,4k+2}>0$, we have
\begin{equation*}
\begin{split}
  S
  &>
  \sum_{l=0}^{2N-4}\left(f_{N-1,l}-f_{N-2,l}\right)
  +
  f_{N-1,2N-3}+f_{N-1,2N-2}
  \\
  &>
  (1-e^{-\xi/2})\sum_{l=0}^{2N-4}f_{N-1,l}
  +
  f_{N-1,2N-3}+f_{N-1,2N-2}
  >
  (1-e^{-\xi/2})\sum_{l=0}^{2N-2}f_{N-1,l}
\end{split}
\end{equation*}
from \eqref{eq:f_N-1_N-2}.
\par
Suppose that $N$ is odd.
We have
\begin{equation*}
\begin{split}
  S
  &=
  \sum_{d=0}^{N-1}\sum_{l=0}^{2d}
  (-1)^{d}f_{d,l}
  \\
  &=
  f_{0,0}
  +
  \sum_{k=0}^{(N-3)/2}
  \left(
    -\sum_{l=0}^{4k+2}f_{2k+1,l}
    +
    \sum_{l=0}^{4k+4}f_{2k+2,l}
  \right)
  \\
  &=
  f_{0,0}
  +
  \sum_{k=0}^{(N-3)/2}
  \left(
    \sum_{l=0}^{4k+2}
    \left(f_{2k+2,l}-f_{2k+1,l}\right)
    +
    f_{2k+2,4k+3}+f_{2k+2,4k+4}
  \right).
\end{split}
\end{equation*}
As in the case where $N$ is even, we have
\begin{equation*}
  S
  >
  \sum_{l=0}^{2N-4}
  \left(f_{N-1,l}-f_{N-2,l}\right)
  +
  f_{N-1,2N-3}+f_{N-1,2N-2}
  >
  (1-e^{-\xi/2})
  \sum_{l=0}^{2N-2}f_{N-1,l}.
\end{equation*}
The proof is complete.
\end{proof}
Now we look for the maximum of $\left\{f_{N-1,l}\bigm| l=0,1,2,\dots,2N-2\right\}$.
\par
We have
\begin{lem}\label{lem:f}
Assume that $1\le l\le2N-2$ and let $\delta$ be a positive real number.
\begin{enumerate}
\item
If $\cosh\left(\frac{l}{N}\xi\right)\ge\cosh(2\xi)-\frac{1}{2}$, then $f_{N-1,l-1}>f_{N-1,l}$.
\item
If $\cosh\left(\frac{l}{N}\xi\right)<\cosh(2\xi)-\frac{1}{2}-\delta$, then there exists $N_0$ such that $f_{N-1,l-1}<f_{N-1,l}$ for $N>N_0$.
\end{enumerate}
\end{lem}
\begin{proof}
First of all, we have
\begin{equation}\label{eq:f^N}
\begin{split}
  \frac{f_{N-1,l}}{f_{N-1,l-1}}
  &=
  4\sinh\left(\frac{(2N-1+l)\xi}{2N}\right)
  \sinh\left(\frac{(2N-1-l)\xi}{2N}\right)
  \\
  &=
  2\cosh\left(2\xi-\frac{\xi}{N}\right)
  -
  2\cosh\left(\frac{l}{N}\xi\right).
\end{split}
\end{equation}
\begin{enumerate}
\item
If $\cosh\left(\frac{l}{N}\xi\right)\ge\cosh(2\xi)-\frac{1}{2}$, it follows that
\begin{equation*}
  \frac{f_{N-1,l}}{f_{N-1,l-1}}
  \le
  2\cosh\left(2\xi-\frac{\xi}{N}\right)
  -
  2\cosh(2\xi)+1
  <1.
\end{equation*}
\item
If $\cosh\left(\frac{l}{N}\xi\right)<\cosh(2\xi)-\frac{1}{2}-\delta$, then it follows that
\begin{equation*}
  \frac{f_{N-1,l}}{f_{N-1,l-1}}
  >
  2\cosh\left(2\xi-\frac{\xi}{N}\right)
  -
  2\cosh(2\xi)+1+2\delta.
\end{equation*}
Therefore if we choose $N_0$ so that
\begin{equation*}
  \cosh(2\xi)-\cosh\left(2\xi-\frac{\xi}{N_0}\right)<\delta,
\end{equation*}
the inequality $f_{N-1,l-1}<f_{N-1,l}$ holds for $N>N_0$.
\end{enumerate}
The proof is complete.
\end{proof}
Now we define
\begin{align*}
  \kappa
  &:=
  \frac{1}{2}\arccosh\left(\frac{3}{2}\right)
  =\frac{1}{2}\log\left(\frac{3+\sqrt{5}}{2}\right),
  \\
  \varphi(\xi)
  &:=
  \arccosh\left(\cosh(2\xi)-\frac{1}{2}\right)
  \\
  &=
  \log
  \left(
    \frac{1}{2}
    \left(
      2\cosh(2\xi)-1+\sqrt{(2\cosh(2\xi)+1)(2\cosh(2\xi)-3)}
    \right)
  \right).
\end{align*}
\begin{rem}
If $\xi>\kappa$, then $\varphi(\xi)$ is real with $0<\varphi(\xi)<2\xi$ because $\cosh(\varphi(\xi))=\cosh(2\xi)-\frac{1}{2}$, which is between $1$ and $\cosh(2\xi)$.
\end{rem}
From Lemma~\ref{lem:f}, we have
\begin{prop}\label{prop:f_max}
If $\xi\le\kappa$, then the maximum of $\left\{f_{N-1,l}\bigm| l=0,1,\dots,2N-2\right\}$ is $f_{N-1,0}$.
Moreover from Corollary~\ref{cor:max_f_d}, $f_{N-1,0}$ is indeed the maximum of $\left\{f_{d,l}\bigm| 0\le l\le2d,0\le d\le N-1\right\}$.
\par
If $\xi>\kappa$, the maximum of $\left\{f_{N-1,l}\bigm| l=0,1,\dots,2N-2\right\}$ is $f_{N-1,\floor{\varphi(\xi)N/\xi}-1}$ for sufficiently large $N$, where $\floor{x}$ is the greatest integer less than or equal to $x$.
Moreover from Corollary~\ref{cor:max_f_d}, $f_{N-1,\floor{\varphi(\xi)N/\xi}-1}$ is indeed the maximum of $\left\{f_{d,l}\bigm| 0\le l\le2d,0\le d\le N-1\right\}$.
\end{prop}
\begin{proof}
If $\xi\le\kappa$, that is, $\cosh(2\xi)-\frac{1}{2}\le1$, then any $l$ satisfies the assumption of (i) in Lemma~\ref{lem:f}.
So $f_{N-1,l}$ is decreasing with respect to $l$ ($0\le l\le 2N-2$) and the maximum is $f_{N-1,0}$.
\par
If $\xi>\kappa$, that is, $\cosh(2\xi)-\frac{1}{2}>1$, then we choose $\delta>0$ such that $\cosh(2\xi)-\frac{1}{2}-\delta>1$.
Now if $N$ is sufficiently large, there exists $l$ such that $\cosh(\frac{l}{N}\xi)<\cosh(2\xi)-\frac{1}{2}-\delta$.
Since $\cosh\varphi(\xi)=\cosh\left(\frac{\varphi(\xi)N}{\xi}\times\frac{\xi}{N}\right)=\cosh(2\xi)-\frac{1}{2}$, we see that
\begin{equation*}
  \cosh
  \left(
    \left\lfloor
      \frac{\varphi(\xi)N}{\xi}
    \right\rfloor
    \times\frac{\xi}{N}
  \right)
  \le
  \cosh(2\xi)-\frac{1}{2}
  <
  \cosh
  \left(
    \left(
      \left\lfloor
        \frac{\varphi(\xi)N}{\xi}
      \right\rfloor
      +1
    \right)
    \times\frac{\xi}{N}
  \right).
\end{equation*}
This means that $\left\lfloor\frac{\varphi(\xi)}{\xi}N\right\rfloor$ is the maximum of all integers $l$ satisfying $\cosh(\frac{l}{N}\xi)\le\cosh(2\xi)-\frac{1}{2}$.
So if $l<\left\lfloor\frac{\varphi(\xi)N}{\xi}\right\rfloor$, then we can choose $\delta'>0$ such that $\cosh\left(\frac{l}{N}\xi\right)<\cosh(2\xi)-\frac{1}{2}-\delta'$ and so $f_{N-1,l-1}<f_{N-1,l}$ from (ii) in Lemma~\ref{lem:f}.
If $l\ge\left\lfloor\frac{\varphi(\xi)N}{\xi}\right\rfloor$, then $\cosh\left(\frac{l}{N}\xi\right)\ge\cosh(2\xi)-\frac{1}{2}$ and so $f_{N-1,l}>f_{N-1,l-1}$ from (i) in Lemma~\ref{lem:f}.
\par
Therefore the maximum is $f_{N-1,\floor{\varphi(\xi)N/\xi}-1}$.
\end{proof}
From Lemma~\ref{lem:S} and Proposition~\ref{prop:f_max} we have the following inequalities when $\xi>\kappa$:
\begin{equation*}
  (1-e^{-\xi/2})f_{N-1,\floor{\varphi(\xi)N/\xi}-1}
  <
  S
  <
  N^2\times
  f_{N-1,\floor{\varphi(\xi)N/\xi}-1}.
\end{equation*}
Here the second inequality holds since there are $\sum_{d=0}^{N-1}(2d+1)=N^2$ terms in $S$.
\par
Taking $\log$ and dividing by $N$ we have
\begin{multline*}
  \frac{\log(1-e^{-\xi/2})}{N}
  +
  \frac{\log\left(f_{N-1,\floor{\varphi(\xi)N/\xi}-1}\right)}{N}
  <
  \frac{\log{S}}{N}
  \\
  <
  \frac{2\log{N}}{N}
  +
  \frac{\log\left(f_{N-1,\floor{\varphi(\xi)N/\xi}-1}\right)}{N}.
\end{multline*}
Since the limits $\lim_{N\to\infty}\frac{\log(1-e^{-\xi/2})}{N}$ and $\lim_{N\to\infty}\frac{2\log{N}}{N}$ vanish, we have from the squeeze theorem
\begin{equation*}
\begin{split}
  &
  \lim_{N\to\infty}\frac{\log{S}}{N}
  \\
  =&
  \lim_{N\to\infty}
  \frac{\log\left(f_{N-1,\floor{\varphi(\xi)N/\xi}-1}\right)}{N}
  \\
  =&
  \lim_{N\to\infty}
  \frac{(2b+1)(N^2-N)\xi}{2N^2}
  +
  \lim_{N\to\infty}
  \frac{1}{N}
  \log\left(2\sinh\left(\frac{(2N-1)\xi}{2N}\right)\right)
  \\
  &+
  \lim_{N\to\infty}
  \sum_{k=1}^{\floor{\varphi(\xi)N/\xi}-1}
  \frac{1}{N}
  \log\left(2\sinh\left(\xi+\frac{(k-1)\xi}{2N}\right)\right)
  \\
  &+
  \lim_{N\to\infty}
  \sum_{k=1}^{\floor{\varphi(\xi)N/\xi}-1}
  \frac{1}{N}
  \log\left(2\sinh\left(\xi-\frac{(k+1)\xi}{2N}\right)\right)
  \\
  =&
  \frac{(2b+1)\xi}{2}
  +
  \frac{2}{\xi}
  \int_{\xi}^{\xi+\varphi(\xi)/2}\log(2\sinh{x})\,dx
  +
  \frac{2}{\xi}
  \int_{\xi-\varphi(\xi)/2}^{\xi}\log(2\sinh{x})\,dx
  \\
  =&
  \frac{(2b+1)\xi}{2}
  +
  \frac{2}{\xi}
  \int_{\xi-\varphi(\xi)/2}^{\xi+\varphi(\xi)/2}\log(2\sinh{x})\,dx
  \\
  =&
  \frac{(2b+1)\xi}{2}
  +
  \frac{2}{\xi}
  \int_{\xi-\varphi(\xi)/2}^{\xi+\varphi(\xi)/2}(\log(1-e^{-2x})+x)\,dx
  \\
  &\text{(putting $y:=e^{-2x}$)}
  \\
  =&
  \frac{(2b+1)\xi}{2}
  -
  \frac{2}{\xi}
  \int_{e^{\varphi(\xi)-2\xi}}^{e^{-\varphi(\xi)-2\xi}}\frac{\log(1-y)}{2y}\,dy
  +
  2\varphi(\xi)
  \\
  =&
  \frac{(2b+1)\xi}{2}
  +
  \frac{1}{\xi}\Li_2(e^{-\varphi(\xi)-2\xi})-\frac{1}{\xi}\Li_2(e^{\varphi(\xi)-2\xi})
  +
  2\varphi(\xi).
\end{split}
\end{equation*}
Since $J_N\left(E^{(2,2b+1)};e^{\xi/N}\right)=\frac{e^{-(2b+1)(N^2-1)\xi/(2N)}}{2\sinh(\xi/2)}S$, we finally have
\begin{equation*}
  \lim_{N\to\infty}
  \frac{\log J_N\left(E^{(2,2b+1)};e^{\xi/N}\right)}{N}
  =
  \frac{1}{\xi}\Li_2(e^{-\varphi(\xi)-2\xi})-\frac{1}{\xi}\Li_2(e^{\varphi(\xi)-2\xi})
  +
  2\varphi(\xi).
\end{equation*}
\section{Representation}\label{sec:rep}
In this section, we introduce a representation $\rho_{u}$ from $\pi_1(S^3\setminus E^{(2,2b+1)})$ to $\SL(2;\C)$.
In fact, it turns out to be a representation to $\SL(2;\R)$.
\par
Put $\FigEight:=S^3\setminus\Int{N(E)}$ and $\FigEightCable:=S^3\setminus\Int{N(E^{(2,2b+1)})}$.
Let $L$ be a knot in a solid torus $D:=D^2\times S^1$ depicted in Figure~\ref{fig:cable_space}.
Put $\Cable:=D\setminus\Int{N(L)}$.
\begin{figure}[H]
\pic{0.4}{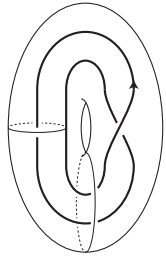}
\caption{A knot $L$ in a solid torus}
\label{fig:cable_space}
\end{figure}
Then $\FigEightCable$ is obtained from $\Cable$ and $\FigEight$ by identifying $\partial{D^2}\times S^1\subset\partial{D}$ with $\partial\FigEight$ so that $\{\text{point}\}\times S^1$ is identified with $\lambda_{E}\mu_{E}^{b}$ and $(\partial{}D^2)\times\{\text{point}\}$ is identified with $\mu_{E}$, where $\lambda_{E}\subset\partial\FigEight$ and $\mu_{E}\subset\partial\FigEight$ are the preferred longitude and the meridian of $E$ (Figure~\ref{fig:decomposition}).
\begin{figure}[H]
\pic{0.4}{cable_space}\hspace{5mm}$\cup$\hspace{5mm}\pic{0.4}{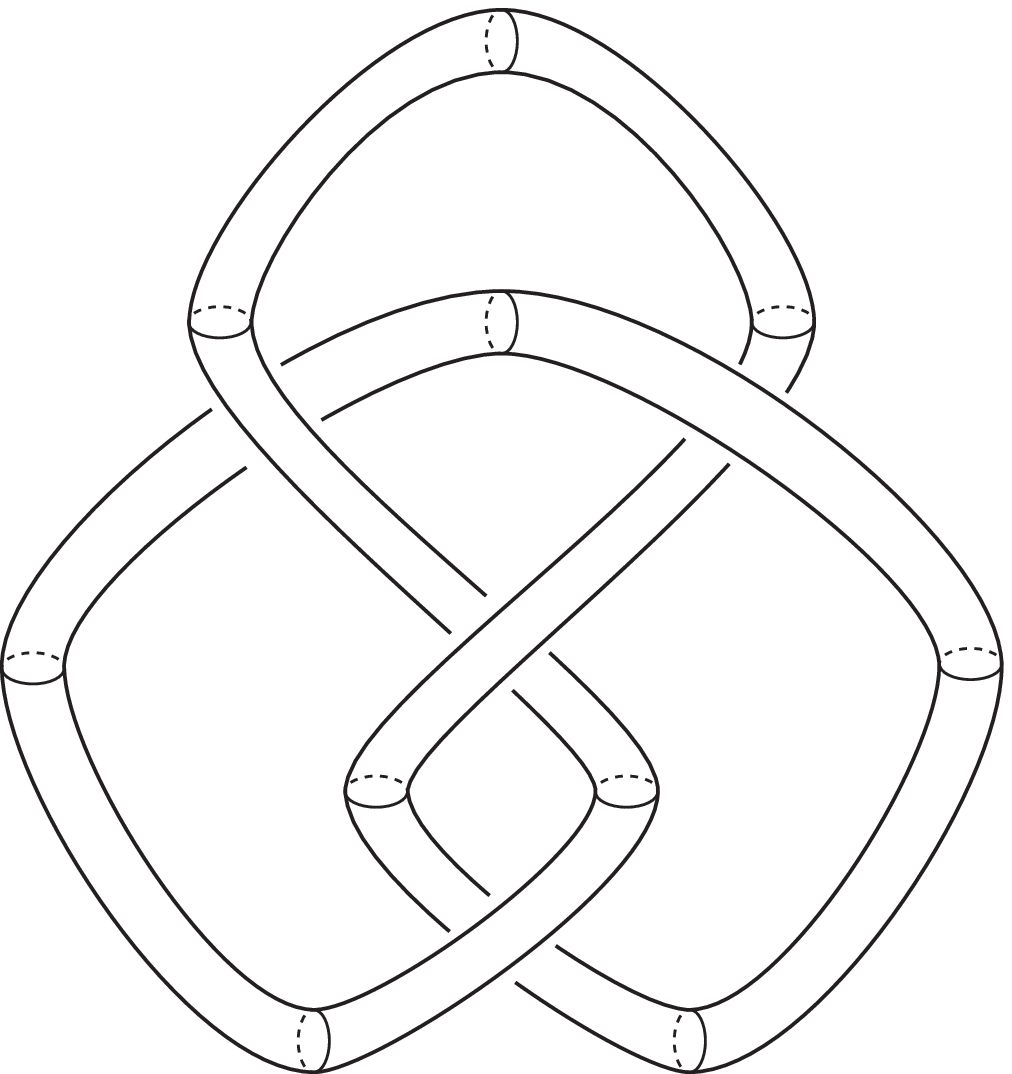}
\caption{$\FigEightCable$ is obtained from $\FigEight$ (left) and a knot $L$ in the solid torus (right).}
\label{fig:decomposition}
\end{figure}
We calculate the fundamental group of $S^3\setminus{E^{(2,2b+1)}}$.
From now on, we identify $\pi_1(S^3\setminus{E})$ with $\pi_1(\FigEight)$, $\pi_1(S^3\setminus{E^{(2,2b+1)}})$ with $\pi_1(\FigEightCable)$, and $\pi_1(D\setminus{L})$ with $\pi_1(\Cable)$, where we choose basepoints in $\partial\FigEight$ that is identified with $\partial{D}$.
If we choose Wirtinger generators as in Figure~\ref{fig:fig8}, then the fundamental group $\pi_1(\FigEight)$ is given as
\begin{equation}\label{eq:pi1_fig8}
  \langle
    x,y\mid xy^{-1}x^{-1}yx=yxy^{-1}x^{-1}y
  \rangle.
\end{equation}
\begin{figure}[H]
\pic{0.4}{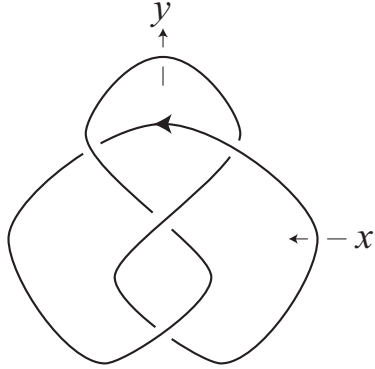}
\caption{Figure-eight knot $E$}
\label{fig:fig8}
\end{figure}
Since $\Cable$ is homeomorphic to the exterior of the $(2,4)$-torus link (Figure~\ref{fig:2_4_torus}), the fundamental group $\pi_1(\Cable)$ is given as
\begin{equation*}
  \pi_1(\Cable)
  =
  \langle
    p,r\mid prpr=rprp
  \rangle,
\end{equation*}
if we choose generators as in Figure~\ref{fig:2_4_torus}.
\begin{figure}
\pic{0.4}{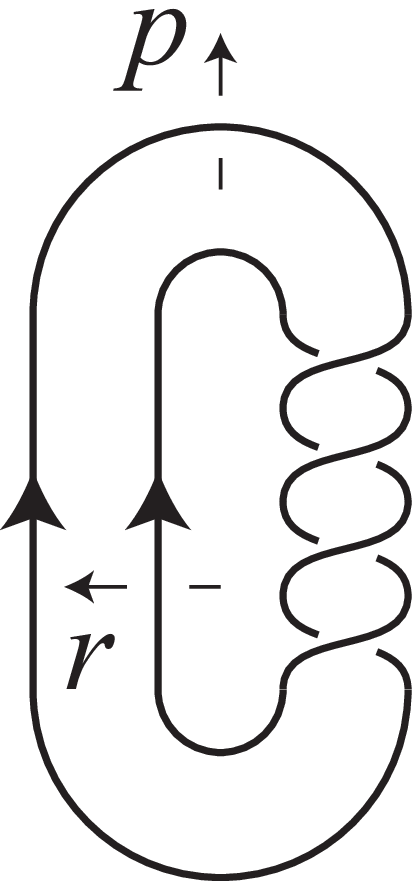}
\hspace{20mm}
\pic{0.4}{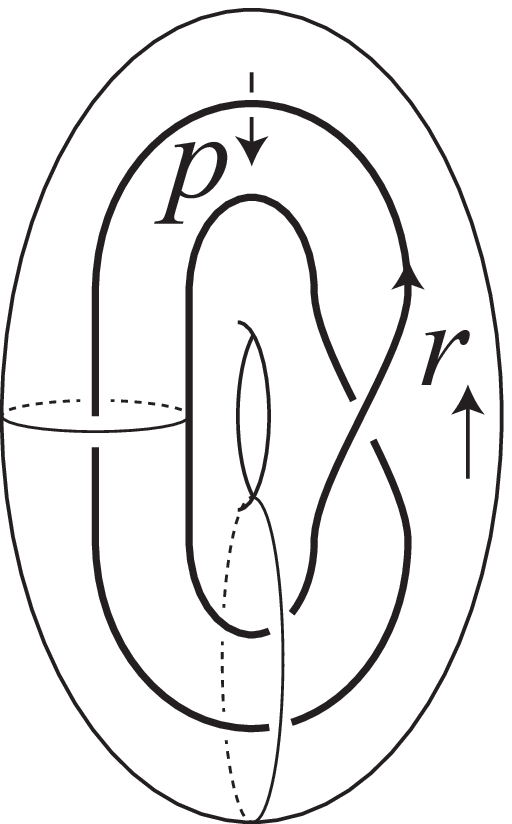}
\caption{$(2,4)$-torus knot (left) and $\Cable$ (right)}
\label{fig:2_4_torus}
\end{figure}
\par
Observe that the meridian of the solid torus $D$ is given as $prpr^{-1}$ (as you can read off by going along the inner circle of the $(2,4)$-torus knot in Figure~\ref{fig:2_4_torus}) and that the longitude $\lambda_E$ of the figure-eight knot is given as $xy^{-1}xyx^{-2}yxy^{-1}x^{-1}$.
By van Kampen's theorem, the fundamental group of $S^3\setminus E^{(2,2b+1)}$ is given as follows:
\begin{equation*}
\begin{split}
  &\pi_1(\FigEightCable)
  \\
  =&
  \left\langle
    x,y\mid xy^{-1}x^{-1}yx=yxy^{-1}x^{-1}y
  \right\rangle
  \bigcup_{\substack{r=\lambda_{E}\mu_{E}^{b}\\prpr^{-1}=\mu_{E}}}
  \left\langle
    p,r\mid prpr=rprp
  \right\rangle
  \\
  =&
  \left\langle
    x,y,p,r
    \mid
    xy^{-1}x^{-1}yx=yxy^{-1}x^{-1}y,
    prpr=rprp,
    x=prpr^{-1},
    \lambda_{E}=rx^{-b}
  \right\rangle,
\end{split}
\end{equation*}
where $\lambda_{E}:=xy^{-1}xyx^{-2}yxy^{-1}x^{-1}$ and we choose $x$ as the meridian $\mu_{E}$ of the figure-eight knot.
Note that we do not need $r$ as a generator.
\par
For a complex number $u$, we define $\rho_{u}\colon\pi_1(\FigEightCable)\to\SL(2;\C)$ as follows.
First define $\rho_{u}\bigm|_{\pi_1(\FigEight)}$ as
\begin{align*}
  \rho_{u}(x)
  &:=
  \begin{pmatrix}
    e^u&1 \\
    0  &e^{-u}
  \end{pmatrix},
  \\
  \rho_{u}(y)
  &:=
  \begin{pmatrix}
    e^{u}&0 \\
    -\delta(u)&e^{-u}
  \end{pmatrix},
\end{align*}
where
\begin{equation*}
  \delta(u)
  :=
  \frac{1}{2}
  \left(
    e^{2u}+e^{-2u}-3+\sqrt{(e^{2u}+e^{-2u}+1)(e^{2u}+e^{-2u}-3)}
  \right).
\end{equation*}
Note that from \cite{Riley:QUAJM31984}, any non-Abelian representation $\pi_1(\FigEight)\to\SL(2;\C)$ is given as above.
The longitude $\lambda_{E}$ of $E$ is sent to
\begin{equation*}
  \rho_{u}(\lambda_E)
  =
  \begin{pmatrix}
    \ell(u) &(e^u+e^{-u})\sqrt{(e^{2u}+e^{-2u}+1)(e^{2u}+e^{-2u}-3)} \\
    0       &\ell(u)^{-1}
  \end{pmatrix},
\end{equation*}
where
\begin{multline*}
  \ell(u)
  \\:=
  \frac{1}{2}(e^{4u}-e^{2u}-2-e^{-2u}+e^{-4u})
  +
  \frac{e^{2u}-e^{-2u}}{2}\sqrt{(e^{2u}+e^{-2u}+1)(e^{2u}+e^{-2u}-3)}.
\end{multline*}
\par
We extend it to $\pi_1(\FigEightCable)$ so that its restriction to $\pi_1(\Cable)$ is Abelian.
We put
\begin{equation*}
  \rho_{u}(p)
  =
  \begin{pmatrix}
    e^{u/2}&\frac{1}{2\cosh(u/2)} \\
    0      &e^{-u/2}
  \end{pmatrix}
\end{equation*}
so that $\rho_{u}(p)^2=\rho_{u}(x)$.
Since
\begin{equation*}
  \rho_{u}(x)^b
  =
  \begin{pmatrix}
    e^{bu}&\frac{e^{bu}-e^{-bu}}{e^{u}-e^{-u}} \\
    0     &e^{-bu}
  \end{pmatrix},
\end{equation*}
we have
\begin{equation*}
\begin{split}
  \rho_{u}(r)
  &=
  \rho_{u}(\lambda_{E})\rho_{u}(x)^b
  \\
  &=
  \begin{pmatrix}
    \ell(u)e^{bu}&
    e^{bu}(e^u+e^{-u})\sqrt{(e^{2u}+e^{-2u}+1)(e^{2u}
     +e^{-2u}-3)}+\ell(u)^{-1}\frac{e^{bu}-e^{-bu}}{e^{u}-e^{-u}}
    \\
    0&\ell(u)^{-1}e^{-bu}
  \end{pmatrix}.
\end{split}
\end{equation*}
We can see that $\rho_{u}(p)$ and $\rho_{u}(r)$ commute and so $\rho_{u}$ is well-defined.
\begin{rem}
Note that $\delta(\kappa)=0$ since $e^{2\kappa}=\frac{3+\sqrt{5}}{2}$, which is a zero of $t-3+t^{-1}$.
Note also that $e^{\kappa}$ is a zero of the Alexander polynomial of $E^{(2,2b+1)}$ that is given as
\begin{equation*}
  \Delta(E^{(2,2b+1)};t)
  =
  (-t^2+3-t^{-2})\times\frac{t^{(2b+1)/2}+t^{-(2b+1)/2}}{t^{1/2}+t^{-1/2}}.
\end{equation*}
\end{rem}
The preferred longitude $\lambda$ of $E^{(2,2b+1)}$ goes twice along $E$.
Since $\rho_{u}\bigm|_{\pi_1(\Cable)}$ is Abelian, it is sent to
\begin{equation*}
  \rho_{u}(\lambda)
  =
  \rho_{u}(\lambda_E)^2
  =
  \begin{pmatrix}
    \ell(u)^2&\ast \\
    0        &\ell(u)^{-2}
  \end{pmatrix}.
\end{equation*}
\begin{rem}
The function $\ell(u)$ satisfies the following equation:
\begin{equation}\label{eq:A_FigEight}
  \ell(u)-(e^{4u}-e^{2u}-2-e^{-2u}+e^{-4u})+\ell(u)^{-1}=0,
\end{equation}
and so $\ell(u)^2$ satisfies
\begin{equation}\label{eq:A_FigEightCable}
  \ell(u)^2+\ell(u)^{-2}
  =
  \left(e^{4u}-e^{2u}-2-e^{-2u}+e^{-4u}\right)^2-2.
\end{equation}
Compare them with the A-polynomial of the figure-eight knot (\cite[Appendix]{Cooper/Culler/Gillet/Long/Shalen:INVEM1994})
\begin{equation*}
  \mathfrak{l}
  -
  (\mathfrak{m}^4-\mathfrak{m}^2-2-\mathfrak{m}^{-2}+\mathfrak{m}^{-4})
  +
  \mathfrak{l}^{-1}
\end{equation*}
and the A-polynomial of $E^{(2,2b+1)}$ (\cite[Example~2.11]{Ni/Zhang:AGT2017}) given by
\begin{equation*}
  \left(1+\mathfrak{m}^{2(2b+1)}\mathfrak{l}\right)
  \left(
    \mathfrak{l}
    -
    \left(
      \mathfrak{m}^8-\mathfrak{m}^4-2-\mathfrak{m}^{-4}+\mathfrak{m}^{-8}
    \right)^2-2
    +
    \mathfrak{l}^{-1}
  \right).
\end{equation*}
\par
Note that the meridian of $E$ ($E^{(2,2b+1)}$, respectively) is given by $x=p^{2}$ ($p$, respectively).
\end{rem}
\begin{rem}
If we use the hyperbolic functions, we have the following expressions for $\ell(u)$.
\begin{align*}
  \ell(u)
  &=
  \cosh(4u)-\cosh(2u)-1+\sinh(2u)\sqrt{(2\cosh(2u)+1)(2\cosh(2u)-3)}
  \\
  &=
  2\cosh^2(2u)-\cosh(2u)-2+\sinh(2u)\sqrt{(2\cosh(2u)+1)(2\cosh(2u)-3)}
  \\
  &=
  \cosh(4u)-\cosh(2u)-1+2\sinh(2u)\sinh\varphi(u).
\end{align*}
Here the last equality follows since $\cosh\varphi(u)=\cosh(2u)-\frac{1}{2}$.
\end{rem}
Since $\cosh(2\kappa)=\frac{3}{2}$ by definition, we have the following equality:
\begin{equation*}
  \ell(\kappa)=1.
\end{equation*}
For $u$ greater than $\kappa$ we have the following lemma.
\begin{lem}\label{lem:ell_real}
If $u>\kappa$, then $\ell(u)$ is a real number with $\ell(u)>1$.
\end{lem}
\begin{proof}
Since $\cosh(2u)>\frac{3}{2}$ when $u>\kappa$, we see that $\ell(u)$ is real.
We also see that
\begin{equation*}
  \ell(u)
  >
  2\cosh^2(2u)-\cosh(2u)-2
  >
  1.
\end{equation*}
\end{proof}
\begin{rem}
If $u$ is real and $u>\kappa$, then the representation $\rho_{u}$ is to $\SL(2;\R)$.
\end{rem}
\section{Chern--Simons invariant}\label{sec:CS}
We will calculate the $\PSL(2;\C)$ Chern--Simons invariant of $\FigEightCable$ associated with the representation $\rho_{u}$.
\par
A practical definition of the Chern--Simons invariant is as follows.
See \cite{Kirk/Klassen:COMMP1993} for the precise definition.
Let $M$ be a three-manifold with boundary $\partial{M}$ a torus.
\par
The $\PSL(2;\C)$ character variety $X(\partial{M})$ of $\pi_1(\partial{M})$ is the quotient set of $\Hom(\pi_1(\partial),\PSL(2;\C))$, where two representations are equivalent if they share the same trace.
It can be described as follows.
Let $\{\mu,\lambda\}$ be a generator set of $\pi_1(\partial{M})\cong\Z^2$ that are presented by oriented simple closed curves on $\partial{M}$.
For a $\PSL(2;\C)$ representation $\gamma$, we can assume
\begin{align*}
  \gamma(\mu)
  &=
  \pm
  \begin{pmatrix}
    e^{2\pi\i g(\gamma)}&\ast \\
    0                   &e^{-2\pi\i g(\gamma)}
  \end{pmatrix},
  \\
  \gamma(\lambda)
  &=
  \pm
  \begin{pmatrix}
    e^{2\pi\i h(\gamma)}&\ast \\
    0                   &e^{-2\pi\i h(\gamma)}
  \end{pmatrix}
\end{align*}
after suitable conjugation.
Then the map $\gamma\mapsto(g(\gamma),h(\gamma))$ gives a one-to-one correspondence between $X(\partial{M})$ and the quotient space $\C^2/H$, where the group $H$ is given as
\begin{equation*}
  H
  :=
  \langle X,Y,b\mid [X,Y]=XbXb=YbYb=b^2=1\rangle,
\end{equation*}
and it acts on $\C^2$ by
\begin{align*}
  X\cdot(x,y)
  &:=
  \left(x+\frac{1}{2},y\right),
  \\
  Y\cdot(x,y)
  &:=
  \left(x,y+\frac{1}{2}\right),
  \\
  b\cdot(x,y)
  &=
  (-x,-y).
\end{align*}
\par
We define $E(\partial{M})$ to be the quotient space of $\C^2\times\C^{\times}$ by the group $H$, where $H$ acts on $E(\partial{M})$ by
\begin{align*}
  X\cdot(x,y;z)
  &:=
  \left(x+\frac{1}{2},y;z\times e^{-4y\pi\i}\right),
  \\
  Y\cdot(x,y;z)
  &:=
  \left(x,y+\frac{1}{2};z\times e^{4x\pi\i}\right),
  \\
  b\cdot(x,y;z)
  &:=
  (-x,-y;z).
\end{align*}
We denote by $[x,y;z]\in E(\partial{M})$ the equivalence class of $(x,y,z)$, that is,
\begin{equation}\label{eq:equiv}
  \begin{cases}
    [x,y;z]
    &=
    \left[x+\frac{1}{2},y;z\times e^{-4y\pi\i}\right],
    \\[3mm]
    [x,y;z]
    &=
    \left[x,y+\frac{1}{2};z\times e^{4x\pi\i}\right],
    \\[3mm]
    [x,y;z]
    &=
    [-x,-y;z].
  \end{cases}
\end{equation}
Then we can see that $q\colon E(\partial{M})\to X(\partial{M})$ ($q\colon[x,y;z]\mapsto[x,y]$) is a $\C^{\times}$-bundle.
\par
Now the Chern--Simons function $\cs_{M}$ is a map from the $\PSL(2,\C)$ character variety $X(M)$ of $\pi_1(M)$ to $E(\partial{M})$ such that the following diagram commutes:
\begin{equation*}
\begin{tikzcd}
  &E(\partial{M})\arrow[d,"q"]
  \\
  X(M)\arrow[ru,"\cs_M"]\arrow[r,"r"]&X(\partial{M}).
\end{tikzcd}
\end{equation*}
Here $r\colon X(M)\to X(\partial{M})$ is the restriction map.
\begin{defn}
If $\cs_{M}([\gamma])=\left[\frac{u}{4\pi\i},\frac{v}{4\pi\i};z\right]$, we put $\CS_{M}([\gamma];u,v):=\frac{\pi\i}{2}\log{z}\pmod{\pi^2\Z}$ and call it the Chern--Simons invariant of $M$ associated with $[\gamma]$ and the representative $\left(\frac{u}{4\pi\i},\frac{v}{4\pi\i}\right)$ of $X(\partial{M})$, where we denote by $[\gamma]\in X(M)$ the equivalence class of the representation $\gamma$.
\end{defn}
To calculate the Chern--Simons invariant, we use the following theorem of Kirk and Klassen \cite{Kirk/Klassen:COMMP1993}.
\begin{thm}[Kirk--Klassen]\label{thm:KK}
Let $M$ a three-manifold with boundary $\partial M$ a torus.
Let $\gamma_t\colon\pi_1(M)\to\PSL(2;\C)$ be a smooth path of representations such that
\begin{align*}
  \gamma_t(\mu_M)
  &=
  \begin{pmatrix}
    e^{u(t)/2}& \ast \\
    0         &e^{-u(t)/2}
  \end{pmatrix},
  \\
  \gamma_t(\lambda_M)
  &=
  \begin{pmatrix}
    e^{v(t)/2}&\ast \\
    0         &e^{-v(t)/2}
  \end{pmatrix}
\end{align*}
up to conjugation, where $\mu_M,\lambda_M\in\pi_1(\partial{M})$ are the meridian and the longitude of $\partial{M}$.
\par
If $\cs_M([\rho_t])=\left[\frac{u(t)}{4\pi\i},\frac{v(t)}{4\pi\i};z(t)\right]$, then we have
\begin{equation*}
  \frac{z(1)}{z(0)}
  =
  \exp
  \left(
    \frac{\i}{2\pi}
    \int_{0}^{1}
    \bigl(
      u(t)\frac{d\,v(t)}{d\,t}-\frac{d\,u(t)}{d\,t}v(t)
    \bigr)\,dt
  \right).
\end{equation*}
\end{thm}
\par
We calculate the Chern--Simons invariant of $\FigEightCable$ by using Theorem~\ref{thm:KK}.
\par
Putting $u_t:=(u-\kappa)t+\kappa$, we define two paths of representations $\alpha_t$ and $\beta_t$ ($0\le t\le1$) as follows.
\begin{align*}
  \alpha_t
  &\colon
  p\mapsto
  \begin{pmatrix}
    e^{t\kappa/2}&0 \\
    0            &e^{-t\kappa/2}
  \end{pmatrix},
  \quad
  x\mapsto
  \begin{pmatrix}
    e^{t\kappa}&0 \\
    0          &e^{-t\kappa}
  \end{pmatrix},
  \quad
  y\mapsto
  \begin{pmatrix}
    e^{t\kappa}&0 \\
    0     &e^{-t\kappa}
  \end{pmatrix},
  \\
  \beta_t
  &\colon
  p\mapsto
  \begin{pmatrix}
    e^{u_t/2}&\frac{1}{2\cosh(u_t/2)} \\
    0        &e^{-u_t/2}
  \end{pmatrix},
  x\mapsto
  \begin{pmatrix}
    e^{u_t}&1 \\
    0      &e^{-u_t}
  \end{pmatrix},
  \quad
  y\mapsto
  \begin{pmatrix}
    e^{u_t}&0 \\
    -d(u_t)&e^{-u_t}
  \end{pmatrix}.
\end{align*}
Note that $\alpha_1$ and $\beta_0$ share the same trace and that $\beta_0$ is upper-triangular.
\par
We can write
\begin{align*}
  \cs_{\FigEightCable}([\alpha_t])
  &:=
  \left[
    \frac{t\kappa}{4\pi\i},0;z(t)
  \right],
  \\\cs_{\FigEightCable}([\beta_t])
  &:=
  \left[
    \frac{u_t}{4\pi\i},\frac{4\log\ell(u_t)}{4\pi\i};w(t)
  \right],
\end{align*}
since
\begin{align*}
  \alpha_t(\lambda)
  &=
  \begin{pmatrix}1&0\\0&1\end{pmatrix},
  \\
  \beta_t(\lambda)
  &=
  \begin{pmatrix}
    \ell(u_t)^2&\ast \\
    0          &\ell(u_t)^{-2}
  \end{pmatrix}.
\end{align*}
Note that since $\ell(u_t)>1$ (Lemma~\ref{lem:ell_real}), $\log\ell(u_t)>0$.
Therefore we have $z(1)=z(0)=1$ and
\begin{equation*}
\begin{split}
  &\frac{w(1)}{w(0)}
  \\
  =&
  \exp
  \left(
    \frac{\i}{2\pi}
    \int_{0}^{1}
    \left(
      u_t\times\frac{4d\,\log\ell(u_t)}{d\,t}
      -
      4\log\ell(u_t)\times\frac{d\,u_t}{d\,t}
    \right)\,dt
  \right)
  \\
  =&
  \exp
  \left(
    \frac{2\i}{\pi}
    \left(
      \Bigl[u_t\log\ell(u_t)\Bigr]_{0}^{1}
      -
      2(u-\kappa)\int_{0}^{1}\log\ell(u_t)\,dt
    \right)
  \right)
  \\
  &\quad\text{(putting $s:=u_t$)}
  \\
  =&
  \exp
  \left(
    \frac{2\i}{\pi}
    \left(
      u\log\ell(u)
      -
      2\int_{\kappa}^{u}\log\ell(s)\,ds
    \right)
  \right).
\end{split}
\end{equation*}
Because $\alpha_1$ and $\beta_0$ share the same trace, $\cs_{\FigEightCable}([\alpha_1])=\cs_{\FigEightCable}([\beta_0])$.
Since
\begin{align*}
  \cs_{\FigEightCable}([\alpha_1])
  &=
  \left[
    \frac{\kappa}{4\pi\i},0;1
  \right],
  \\
  \intertext{and}
  \cs_{\FigEightCable}(\beta_0)
  &=
  \left[
    \frac{\kappa}{4\pi\i},\frac{4\log\ell(\kappa)}{4\pi\i};w(0)
  \right],
\end{align*}
we have
\begin{equation*}
  \left[
    \frac{\kappa}{4\pi\i},0;1
  \right]
  =
  \left[
    \frac{\kappa}{4\pi\i},\frac{4\log\ell(\kappa)}{4\pi\i};w(0)
  \right]
  =
  \left[
    \frac{\kappa}{4\pi\i},0;w(0)
  \right]
\end{equation*}
since $\ell(\kappa)=1$.
So we have $w(0)=1$.
\par
Therefore we have
\begin{equation*}
  w(1)
  =
  \exp
  \left(
    \frac{2\i}{\pi}
    \left(
      u\log\ell(u)
      -
      2\int_{\kappa}^{u}\log\ell(s)\,ds
    \right)
  \right)
\end{equation*}
and so
\begin{equation*}
  \CS_{\FigEightCable}\bigl([\rho_{u}];u,v(u)\bigr)
  =
  2\int_{\kappa}^{u}\log\ell(s)\,ds
  -u\log\ell(u),
\end{equation*}
if we put $v(u):=4\log\ell(u)$.
\par
Now we put
\begin{equation*}
  S(u)
  :=
  \Li_2(e^{-\varphi(u)-2u})-\Li_2(e^{\varphi(u)-2u})+2u\varphi(u).
\end{equation*}
We will show that $\CS_{\FigEightCable}\bigl([\rho_{u}];u,v(u)\bigr)=S(u)-u\log\ell(u)$, that is,
\begin{equation}\label{eq:S_ell}
  S(u)=2\int_{\kappa}^{u}\log\ell(s)\,ds.
\end{equation}
\par
Since $e^{\varphi(u)}+e^{-\varphi(u)}=e^{2u}+e^{-2u}-1=2\cosh(2u)-1$, we have
\begin{equation*}
\begin{split}
  \ell(u)
  &=
  \frac{1}{2}
  \left(
    (e^{2u}+e^{-2u})^2-(e^{2u}+e^{-2u})-4
  \right)
  +
  \sinh(2u)\sqrt{(2\cosh(2u)+1)(2\cosh(2u)-3)}
  \\
  &=
  \cosh(4u)-\cosh(2u)-1+\sin(2u)\sqrt{(2\cosh(2u)+1)(2\cosh(2u)-3)}.
\end{split}
\end{equation*}
Therefore we have
\begin{equation}\label{eq:S1}
\begin{split}
  &\exp\left(\frac{d}{d\,u}\left(2\int_{\kappa}^{u}\log\ell(s)\,ds\right)\right)
  \\
  =&
  \left(
    \cosh(4u)-\cosh(2u)-1+\sin(2u)\sqrt{(2\cosh(2u)+1)(2\cosh(2u)-3)}
  \right)^2.
\end{split}
\end{equation}
\par
On the other hand, we have
\begin{equation*}
\begin{split}
  &\frac{d\,S(u)}{d\,u}
  \\
  =&
  \frac{d}{d\,u}\Li_2(e^{-\varphi(u)-2u})
  -
  \frac{d}{d\,u}\Li_2(e^{\varphi(u)-2u})
  +
  2\varphi(u)+2u\varphi'(u)
  \\
  =&
  -
  \frac{\log(1-e^{-\varphi(u)-2u})}
       {e^{-\varphi(u)-2u}}\times(-\varphi'(u)-2)e^{-\varphi(u)-2u}
  \\
  &+
  \frac{\log(1-e^{-\varphi(u)-2u})}
       {e^{\varphi(u)-2u}}\times(\varphi'(u)-2)e^{\varphi(u)-2u}
  +
  2\varphi(u)+2u\varphi'(u)
  \\
  =&
  -
  (-\varphi'(u)-2)\log(1-e^{-\varphi(u)-2u})
  +
  (\varphi'(u)-2)\log(1-e^{\varphi(u)-2u})
  +
  2\varphi(u)+2u\varphi'(u)
  \\
  =&
  \varphi'(u)\log(1-e^{-\varphi(u)-2u})(1-e^{\varphi(u)-2u})
  +
  2\log\frac{1-e^{-\varphi(u)-2u}}{1-e^{\varphi(u)-2u}}
  +
  2\varphi(u)+2u\varphi'(u)
  \\
  =&
  \varphi'(u)\log(1-e^{-\varphi(u)-2u}-e^{\varphi(u)-2u}+e^{-4u})
  +
  2\log\frac{1-e^{-\varphi(u)-2u}}{1-e^{\varphi(u)-2u}}
  +
  2\varphi(u)
  +
  2u\varphi'(u).
\end{split}
\end{equation*}
Since
\begin{equation*}
  1-e^{-\varphi(u)-2u}-e^{\varphi(u)-2u}+e^{-4u}
  =
  e^{-2u}
  \left(
    e^{2u}+e^{-2u}-e^{-\varphi(u)}-e^{\varphi(u)}
  \right)
  =
  e^{-2u},
\end{equation*}
we have
\begin{equation*}
  \frac{d\,S(u)}{d\,u}
  =
  2\log\frac{1-e^{-\varphi(u)-2u}}{1-e^{\varphi(u)-2u}}
  +
  2\varphi(u)
  =
  2\log\frac{e^{\varphi(u)}-e^{-2u}}{1-e^{\varphi(u)-2u}}.
\end{equation*}
Now we calculate
\begin{equation*}
\begin{split}
  &
  \frac{e^{\varphi(u)}-e^{-2u}}{1-e^{\varphi(u)-2u}}
  \\
  =&
  \frac{\left(e^{\varphi(u)}-e^{-2u}\right)\left(e^{\varphi(u)+2u}-1\right)}
       {\left(1-e^{\varphi(u)-2u}\right)\left(e^{\varphi(u)+2u}-1\right)}
  \\
  =&
  \frac{e^{2\varphi(u)+2u}+e^{-2u}-2e^{\varphi(u)}}
       {e^{\varphi(u)+2u}+e^{\varphi(u)-2u}-e^{2\varphi(u)}-1}
  \\
  =&
  \frac{e^{\varphi(u)+2u}+e^{-\varphi(u)-2u}-2}
       {e^{2u}+e^{-2u}-e^{\varphi(u)}-e^{-\varphi(u)}}
  \\
  =&
  e^{\varphi(u)+2u}+e^{-\varphi(u)-2u}-2.
  \\
  =&
  \frac{1}{2}e^{4u}+\frac{1}{2}-\frac{1}{2}e^{2u}
  +
  \frac{1}{2}e^{2u}\sqrt{(2\cosh(2u)+1)(2\cosh(2u)-3)}
  \\
  &+
  \frac{1}{2}+\frac{1}{2}e^{-4u}-\frac{1}{2}e^{-2u}
  -
  \frac{1}{2}e^{-2u}\sqrt{(2\cosh(2u)+1)(2\cosh(2u)-3)}
  -2
  \\
  =&
  \cosh(4u)-\cosh(2u)-1+\sinh(2u)\sqrt{(2\cosh(2u)+1)(2\cosh(2u)-3)}.
\end{split}
\end{equation*}
Here we use the following equalities:
\begin{align*}
  e^{\varphi(u)}
  &=
  \frac{1}{2}
  \left(
    e^{2u}+e^{-2u}-1+\sqrt{(2\cosh(2u)+1)(2\cosh(2u)-3)}
  \right),
  \\
  e^{-\varphi(u)}
  &=
  \frac{1}{2}
  \left(
    e^{2u}+e^{-2u}-1-\sqrt{(2\cosh(2u)+1)(2\cosh(2u)-3)}
  \right).
\end{align*}
Therefore we conclude
\begin{equation*}
\begin{split}
  &\exp\left(\frac{d\,S(u)}{d\,u}\right)
  \\
  =&
  \left(
    \cosh(4u)-\cosh(2u)-1+\sinh(2u)\sqrt{(2\cosh(2u)+1)(2\cosh(2u)-3)}
  \right)^2.
\end{split}
\end{equation*}
\par
From \eqref{eq:S1}, we have
\begin{equation*}
  \exp\left(\frac{d\,S(u)}{d\,u}\right)
  =
  \exp\left(\frac{d}{d\,u}\left(2\int_{\kappa}^{u}\log\ell(s)\,ds\right)\right).
\end{equation*}
Since $S(\kappa)=0$, \eqref{eq:S_ell} follows.
\par
Thus, we obtain the following theorem.
\begin{thm}\label{thm:main}
Let $\xi$ be a positive real number with $\xi>\frac{1}{2}\arccosh\left(\frac{3}{2}\right)$, then we have
\begin{equation*}
  \xi\lim_{N\to\infty}
  \frac{\log J_{N}\left(E^{(2,2b+1)};e^{\xi/N}\right)}{N}
  =
  S(\xi),
\end{equation*}
where we put
\begin{equation*}
  S(\xi)
  :=
  \Li_2\left(e^{-\varphi(\xi)-2\xi}\right)
  -
  \Li_2\left(e^{\varphi(\xi)-2\xi}\right)
  +2\xi\varphi(\xi)
\end{equation*}
with
\begin{equation*}
  \varphi(\xi)
  :=
  \arccosh\left(\cosh(2\xi)-\frac{1}{2}\right).
\end{equation*}
We have also shown that $S(\xi)$ defines the Chern--Simons invariant $\CS_{\FigEightCable}([\rho_{\xi}];\xi,v(\xi))$ of the knot exterior $\FigEightCable=S^3\setminus\Int{N(E^{(2,2b+1)})}$ associated with $\rho_{\xi}$ and $(\xi,v(\xi))$ defined in Section~\ref{sec:rep} in the following way:
\begin{equation*}
  \CS_{\FigEightCable}([\rho_{\xi}];\xi,v(\xi))
  =
  S(\xi)-\frac{\xi v(\xi)}{4}.
\end{equation*}
Here the meridian $p\in\pi_1(\FigEightCable)$ is sent to $\begin{pmatrix}e^{\xi/2}&1\\0&e^{-\xi/2}\end{pmatrix}$ and the longitude $\lambda\in\pi_1(\FigEightCable)$ is sent to $\begin{pmatrix}-e^{v(\xi)/2}&\ast\\0&-e^{-v(\xi)/2}\end{pmatrix}$.
\end{thm}
Compare this with the result about the figure-eight knot.
In \cite[Theorem~8.1]{Murakami:KYUMJ2004} and \cite[Theorem~6.9]{Murakami:Novosibirsk}, the first author proved the following theorem.
\begin{thm}[\cite{Murakami:KYUMJ2004}]\label{thm:fig8}
Let $\eta$ be a real number with $\eta>2\kappa$, then we have
\begin{equation*}
  \eta\lim_{N\to\infty}\frac{J_N(E;e^{\eta/N})}{N}
  =
  S(\eta/2).
\end{equation*}
\end{thm}
\par
We can express the Chern--Simons invariant of $\FigEight$ in terms of $S(\eta/2)$ as follows.
\par
If we define $\sigma_{u}:=\rho_{u/2}\bigm|_{\pi_1(\FigEight)}$, then
\begin{align*}
  \sigma_{u}(x)
  &=
  \begin{pmatrix}
    e^{u/2}&1 \\
    0      &e^{-u/2}
  \end{pmatrix},
  \\
  \sigma_{u}(y)
  &=
  \begin{pmatrix}
    e^{u/2}&0 \\
    -d(u/2)&e^{-u/2}
  \end{pmatrix},
\end{align*}
where $x$ and $y$ are generators of $\pi_1(\FigEight)$ given in \eqref{eq:pi1_fig8}.
The longitude $\lambda_E$ of $\FigEight$ is sent to
\begin{equation*}
  \sigma_{u}(\lambda_E)
  =
  \begin{pmatrix}
    \ell(u/2)&\ast \\
    0        &\ell(u/2)^{-1}
  \end{pmatrix}.
\end{equation*}
\par
By a calculation similar to $E^{(2,2b+1)}$, the Chern--Simons function $\cs_{\FigEight}$ is given as
\begin{equation*}
  \cs_{\FigEight}(\sigma_{u})
  =
  \left[
    \frac{u}{4\pi\i},
    \frac{2\log\ell(u/2)}{4\pi\i},
    \exp\left(\frac{2}{\pi\i}\CS([\sigma_{u}];u,2\log\ell(u/2))\right)
  \right]
\end{equation*}
with
\begin{equation*}
\begin{split}
  \CS_{\FigEight}([\sigma_{u}];u,\vE(u))
  &=
  \Li_2(e^{-\varphi(u/2)-u})
  -
  \Li_2(e^{\varphi(u/2)-u})
  +u\varphi(u/2)
  -\frac{1}{4}u\vE(u)
  \\
  &=
  S(u/2)-\frac{u\vE(u)}{4},
\end{split}
\end{equation*}
where $\vE(u):=2\log\ell(u/2)$.
So we have
\begin{equation*}
  \CS_{\FigEight}([\sigma_{\eta}];\eta,\vE(\eta))
  =
  S(\eta/2)-\frac{\eta\vE(\eta)}{4}.
\end{equation*}
\par
Since $\vE(u)=\frac{1}{2}v(u/2)$, we have
\begin{equation*}
  \CS_{\FigEightCable}([\rho_{\xi}];\xi,v(\xi))
  =
  \CS_{\FigEight}
  \left(
    \left[\rho_{\xi}\bigm|_{\pi_1(\FigEight)}\right];2\xi,2\vE(2\xi)
  \right).
\end{equation*}
\bibliography{mrabbrev,hitoshi}
\bibliographystyle{amsplain}
\end{document}